%% file: laurent.tex
\documentclass[10pt, a4paper]{amsart}

\pdfoutput=1

\usepackage[T1]{fontenc}

\usepackage{mathrsfs, amsmath, amsthm, amssymb, stmaryrd, enumerate, bbm, mathdots,graphicx}%
\usepackage{tikz}%
\usetikzlibrary{knots}
\usepackage[linktocpage]{hyperref}%
\usepackage[all]{xy}%
\xyoption{2cell}%
\UseAllTwocells%

\newdir{t>}{{\UseTips\dir{>}}}
\newdir{ >}{{}*!/-5pt/@{>}}

\DeclareMathOperator{\id}{id}

\DeclareMathOperator{\Mod}{\mathbf{Mod}}

\DeclareMathOperator{\fgp}{fgp}

\DeclareMathOperator{\colim}{colim}

\DeclareMathOperator{\Set}{\mathbf{Set}}

\DeclareMathOperator{\CAlg}{\mathbf{CAlg}}
\DeclareMathOperator{\CRing}{\mathbf{CRing}}

\newcommand{\ca}[1]{\mathscr{#1}}

\newcommand{\ten}[1]{\mathop{{\otimes}_{#1}}}

\newcommand{\defl}{\mathrel{\mathop:}=}

\newcommand{\downmapsto}{\rotatebox[origin=c]{-90}{$\mapsto$}\mkern2mu}

\theoremstyle{plain}
\newtheorem{thm}{Theorem}[section]
\newtheorem*{thm*}{Theorem}
\newtheorem{prop}[thm]{Proposition}
\newtheorem{lemma}[thm]{Lemma}

\theoremstyle{definition}

\newtheorem{dfn}[thm]{Definition}

\newtheorem{question}[thm]{Question}

\newtheoremstyle{citing}{}{}{\itshape}{}{\bfseries}{.}{ }{\thmnote{#3}}
\theoremstyle{citing}
\newtheorem{cit}{}

\newtheoremstyle{citingdfn}{}{}{}{}{\bfseries}{.}{ }{\thmnote{#3}}
\theoremstyle{citingdfn}

\numberwithin{equation}{section}


\author{Daniel Sch\"appi}
\thanks{This research was supported by the DFG grant: SFB 1085 ``Higher invariants.''}
\address{Fakult{\"a}t f{\"u}r Mathematik,
Universit{\"a}t Regensburg,
93040 Regensburg,
Germany}
\email{daniel.schaeppi@ur.de}

\title{Formal Laurent series rings and the Hermite ring conjecture}

\begin{document}

\begin{abstract}

 We study the question if projective modules over formal Laurent series rings are extended. We relate this question to the Bass--Quillen conjecture for commutative regular local rings and to the Hermite ring conjecture for all commutative local rings.
 
 Using our result about projective modules over formal Laurent series rings, we prove a reduction step for the Hermite ring conjecture. We show that the Hermite ring conjecture holds for all commutative local rings if and only if it holds for complete intersection rings which are also unique factorization domains.
\end{abstract}

\maketitle

\tableofcontents

\input{introduction}
\input{formal}
\input{universal}
\input{reduction}

\bibliographystyle{amsalpha}
\bibliography{laurent}

\end{document}

%% file: introduction.tex
\section{Introduction}\label{section:introduction}

The Bass--Quillen conjecture is the conjecture that for a commutative regular ring $R$, all finitely generated projective modules over the polynomial ring $R[t_1, \ldots, t_n]$ are extended from $R$. Lennart Meier asked if the analogous statement holds for the ring $R(\!(t_1, \ldots, t_n)\!)$ of formal Laurent series over $R$ \cite{MEIER}.

Thanks to Quillen's Patching Theorem \cite[Theorem~1]{QUILLEN}, the Bass--Quillen conjecture reduces to the assertion that for all commutative regular local rings $R$, all finitely generated projective $R[t]$-modules are free. The theorem below shows that this is \emph{equivalent} to the corresponding statement for the ring $R(\!(t)\!)$ of formal Laurent series in one variable. We write $R \langle t \rangle$ for the localization of $R[t]$ at the set of monic polynomials.

\begin{thm}\label{thm:regular_local_equivalence}
 Let $R$ be a commutative regular local ring. Then the following are equivalent:
 \begin{enumerate}
 \item[(i)] All finitely generated projective $R[t]$-modules are free;
 \item[(ii)] All finitely generated projective $R[t,t^{-1}]$-modules are free;
 \item[(iii)] All finitely generated projective $R\langle t\rangle$-modules are free;
 \item[(iv)] All finitely generated projective $R(\!(t)\!)$-modules are free.
 \end{enumerate}
\end{thm}

 An example due to Swan shows that projective $R(\!(t)\!)$-modules need not be extended from $R$ if the regular ring $R$ is not local (see Proposition~\ref{prop:Swan_example}). 
 
 The regularity assumption is used to show that all finitely generated projective modules over the various rings occurring in Theorem~\ref{thm:regular_local_equivalence} are stably free. Recall that a finitely generated projective $A$-module $P$ is \emph{stably free} if there exist $m,n \in \mathbb{N}$  such that $P \oplus A^m \cong A^n$. Following Lam, we call a ring \emph{Hermite} if all finitely generated stably free modules are free (see \cite[Definition~I.4.6]{LAM}). Theorem~\ref{thm:regular_local_equivalence} is a consequence of the following result, which is valid for \emph{all} local rings.
 
 \begin{thm}\label{thm:Hermite_equivalence}
 Let $R$ be a commutative local ring. Then the following are equivalent:
 \begin{enumerate}
 \item[(i)] The ring $R[t]$ is Hermite;
 \item[(ii)] The ring $R[t,t^{-1}]$ is Hermite;
 \item[(iii)] The ring $R \langle t \rangle$ is Hermite;
 \item[(iv)] The ring $R(\!(t)\!)$ is Hermite.
 \end{enumerate}
 \end{thm}
 
 We note that the equivalence $(i) \Leftrightarrow (iii)$ above already appears as Theorem~A in \cite{BHATWADEKAR_RAO}. Our proof of Theorem~\ref{thm:Hermite_equivalence} also gives a novel proof of \cite[Theorem~A]{BHATWADEKAR_RAO}.
 
 Lam posed the following question about commutative Hermite rings (see \cite[p.~ix]{LAM_ORIGINAL}): if $R$ is a commutative Hermite ring, is $R[t]$ necessarily Hermite? This is known as the Hermite ring conjecture $(\mathrm{H})$. This is a far-reaching generalization of the Bass--Quillen conjecture; by Quillen's Patching Theorem, this question has an affirmative answer if and only if the equivalent conditions of Theorem~\ref{thm:Hermite_equivalence} hold for \emph{all} commutative local rings.
 
 As an application of Theorem~\ref{thm:Hermite_equivalence}, we show that the Hermite ring conjecture can be reduced to a question about a countable set of relatively ``nice'' rings.
 
 \begin{dfn}\label{dfn:test_set_for_H}
  Given a class (or set) of commutative rings $\ca{C}$, we write $(\mathrm{H}[\ca{C}])$ for the assertion:
  \begin{center}
  For all $A \in \ca{C}$, the ring $A[t]$ is Hermite.
  \end{center}
  We let $\ca{H}$ denote the class of all commutative Hermite rings. We call $\ca{C}$ a \emph{test class for $(\mathrm{H})$} (respectively a \emph{test set for $(\mathrm{H})$}) if $\ca{C}$ is a class (respectively a set) of commutative rings and the implication
  \[
  (\mathrm{H}[\ca{C}]) \Rightarrow (\mathrm{H}[\ca{H}])
  \]
  holds.
 \end{dfn}
 
 For example, Quillen's Patching Theorem implies that the class $\ca{L}$ of commutative local rings is a test class for $(\mathrm{H})$.
 
 Recall that an ideal in a commutative ring is called a \emph{complete intersection ideal} if it is generated by a regular sequence. We call a commutative ring $A$ a \emph{global complete intersection ring} if there exists a polynomial ring $\mathbb{Z}[t_1, \ldots, t_n]$ and a complete intersection ideal $I \subseteq \mathbb{Z}[t_1,\ldots,t_n]$ such that $A \cong \mathbb{Z}[t_1, \ldots, t_n] \slash I$.
 
 \begin{thm}\label{thm:test_set_of_complete_intersections}
 There exists a countable test set $\ca{T}$ for $(\mathrm{H})$ with the following properties:
 \begin{enumerate}
 \item[(i)] Each $A \in \ca{T}$ is $\mathbb{N}$-graded, with degree $0$ part $A_0=\mathbb{Z}$;
 \item[(ii)] Each $A \in \ca{T}$ is a global complete intersection ring;
 \item[(iii)] Each $A \in \ca{T}$ is a unique factorization domain. 
 \end{enumerate}
 \end{thm}
 
 If $(\mathrm{H}[\ca{H}])$ holds, then for each $\mathbb{N}$-graded commutative ring $A=A_0 \oplus A_1 \oplus \ldots$, each stably free $A$-module is extended from $A_0$ by the Swan--Weibel Homotopy Trick (see Proposition~\ref{prop:graded_homotopy_trick}). In particular, $(\mathrm{H}[\ca{H}])$ implies that $A$ is Hermite if $A_0 \cong \mathbb{Z}$. Thus the reverse implication
 \[
 (\mathrm{H}[\ca{H}]) \Rightarrow (H[\ca{T}])
 \]
 also holds for the set $\ca{T}$ of Theorem~\ref{thm:test_set_of_complete_intersections}.
 
 A more conservative conjecture than $(\mathrm{H}[\ca{H}])$ is that $(\mathrm{H}[\ca{H}_{\mathbb{Q}}])$ holds, where $\ca{H}_{\mathbb{Q}}$ is the class of commutative Hermite rings which are also $\mathbb{Q}$-algebras. This is closely related to a problem posed by Suslin, see \cite[Problem~4]{SUSLIN}. One can also find test sets for $(\mathrm{H}[\ca{H}_{\mathbb{Q}}])$ and for Suslin's Problem (see \S \ref{section:reduction}). By contemplating the simplest rings in the respective test sets, one is naturally led to the following question.
 
 \begin{question}
 Let $r \geq 2$ be a natural number and let $A$ be either one of the rings
 \[
 \mathbb{Z}[x_0, \ldots, x_r, y_0, \ldots, y_r] \slash \textstyle \sum_{i=0}^r x_i y_i \quad \text{or}  \quad \mathbb{Q}[x_0, \ldots, x_r, y_0, \ldots, y_r] \slash \textstyle \sum_{i=0}^r x_i y_i \smash{\rlap{.}}
 \]
 Is every finitely generated stably free $A[t]$-module free?
 \end{question}
 
 \section*{Acknowledgements}
 I thank Niko Naumann for suggesting several improvements to the exposition.

%% file: formal.tex
\section{Unimodular rows over formal Laurent series rings}\label{section:formal}

 We start with a few recollections. To show that all finitely generated stably free modules over a commutative ring $A$ are free, it suffices to prove this for finitely generated stably free modules of type $1$, that is, modules $P$ such that $P \oplus A \cong A^{r+1}$ for some $r \in \mathbb{N}$. Indeed, applying this $\ell$ times we find that a module $P$ such that $P\oplus A^{\ell} \cong A^{r+\ell}$ has to be free. The stably free modules of type $1$ are precisely the kernels of epimorphisms
 \[
 a =(a_0, \ldots, a_r) \colon A^{r+1} \rightarrow A \smash{\rlap{,}}
 \]
 which are called \emph{unimodular rows} of length $r+1$. Note that such an epimorphism is necessarily split. We denote the set of unimodular rows of length $r+1$ by $\mathrm{Um}_{r+1}(A)$. The kernel of a unimodular row is free if and only if the unimodular row can be completed to an invertible matrix over $A$. We call such unimodular rows \emph{completable}.
 
 Matrix multiplication defines an action of $\mathrm{GL}_{r+1}(A)$ on $\mathrm{Um}_{r+1}(A)$. Given two unimodular rows $a$ and $b$ of length $r+1$, we write $a \sim b$ if there exists a matrix $M \in \mathrm{GL}_{r+1}(A)$ such that $a =b \cdot M$. Clearly $a$ is completable if and only if $a \sim (1,0,\ldots,0)$. In particular, all stably free $A$-modules are free if and only if the group $\mathrm{GL}_{r+1}(A)$ acts transitively on $\mathrm{Um}_{r+1}(A)$ for all $r \geq 1$. It is well-known that this is always the case if $r=1$, so we may restrict attention to the case $r \geq 2$.

 Recall that a polynomial $\sum_{i=0}^n a_i t^i$ over a commutative local ring $(R, \mathfrak{m})$ is called a \emph{Weierstrass polynomial} if it is monic and $a_i \in \mathfrak{m}$ for $0 \leq i <n$. We will use the following basic fact about Weierstrass polynomials.
 
 \begin{lemma}\label{lemma:top-bottom}
 Let $(R, \mathfrak{m})$ be a commutative local ring and let $f=\sum_{i=0}^n b_i t^i$ be a polynomial in $R[t]$ such that $b_0$ is a unit. Then for any Weierstrass polynomial $g$, the ploynomials $f$ and $g$ are comaximal (that is, the ideal generated by $f$ and $g$ is the whole ring).
 \end{lemma}
 
 \begin{proof}
 This is part of the Top-Bottom Lemma \cite[Lemma~IV.5.3]{LAM}.
 \end{proof}
 
 The crucial step in the proof of Theorem~\ref{thm:Hermite_equivalence} is the observation that the $\mathrm{GL}_{r+1}$-orbit of a unimodular row over the formal Laurent series ring always contains a unimodular row over the Laurent polynomial ring. This is the content of the following elementary lemma, which is proved in three steps. First one notes that the orbit contains a unimodular row whose reduction modulo the maximal ideal consists of entries which are all $0$ except for a single entry $1$. In the second step, we use an auxiliary matrix in $\mathrm{GL}_{r+1}\bigl(R\llbracket t \rrbracket\bigr)$ whose (mulitplicative) action on the unimodular row has the effect of an addition in each entry. This is the key step in the proof and yields a row whose entries are polynomials. The final step is the proof that this row is already unimodular over the ring of Laurent polynomials. This last step relies on Lemma~\ref{lemma:top-bottom} above.
 
 \begin{lemma}\label{lemma:reduction_step}
 Let $(R,\mathfrak{m})$ be a local ring. Let $r \geq 2$ and let $(x_0, \ldots, x_r)$ be a unimodular row over $R(\!(t)\!)$. Then the $\mathrm{GL}_{r+1}$-orbit of $(x_0, \ldots, x_r)$ contains the image of a unimodular row $(p_0, \ldots, p_r)$ over $R[t,t^{-1}]$ with the additional properties that $p_0$ is a Weierstrass polynomial and $p_i$ is a polynomial in $R[t]$ for $i \geq 1$.
 \end{lemma}

\begin{proof}
 Let $(x_0, \ldots, x_r) \in \mathrm{Um}_{r+1}\bigl( R (\!(t)\!) \bigr)$ be a unimodular row. We denote the surjection $ R(\!(t)\!) \rightarrow R \slash \mathfrak{m}  (\!(t)\!)$, $\sum_{i \geq i_0} a_i t^i \mapsto \sum_{i \geq i_0} \overline{a_i} t^i$ by $a \mapsto \overline{a}$. By applying this surjection we obtain a unimodular row $(\overline{x_0}, \ldots, \overline{x_r})$ over the field $R \slash \mathfrak{m} (\!(t)\!)$. Via an elementary matrix $E$ we can change $(\overline{x_0}, \ldots, \overline{x_r})$ to the row $(0,1,0, \ldots, 0)$. Since the matrix $E$ is elementary, it can be lifted to $R(\!(t)\!)$, so we can without loss of generality assume that $\overline{x_0}=0$, $\overline{x_1}=1$, and $\overline{x_i}=0$ for $i \geq 2$.
 
Next, we want to show that there exists a unimodular row in the same $\mathrm{GL}_{r+1}$-orbit whose entries are polynomials. We choose a unimodular row $(y_0, \ldots, y_r) \in \mathrm{Um}_{r+1}\bigl( R (\!(t)\!) \bigr)$ such that the equation $\sum_{i=0}^r x_i y_i=1$ holds. After multiplying both $(x_0, \ldots, x_r)$ and $(y_0, \ldots, y_r)$ by a sufficiently high power of $t$, we can assume that each $x_i$ and each $y_j$ lies in the ring $R \llbracket t \rrbracket$ of formal power series over $R$ and $\sum_{i=0}^r x_i y_i=t^k$ for some $k \in \mathbb{N}$. By construction, we still have $\overline{x_0}=0$ in $R \slash \mathfrak{m} \llbracket t \rrbracket$, that is, all coefficients of $x_0$ lie in $\mathfrak{m}$.
 
We now construct an auxiliary matrix which has the effect of adding (respectively subtracting) terms of sufficiently high degree to each entry of the unimodular row. Let $d_0, \ldots, d_r$ be arbitrary elements of $R \llbracket t \rrbracket$ and let $I_{r+1} \in \mathrm{GL}_{r+1}\bigl(R \llbracket t \rrbracket \bigr) $ be the identity matrix. We claim that the matrix
 \[
 M_{d_0, \ldots, d_r} \defl I_{r+1} +
 \begin{pmatrix}
 y_0 & y_0 & \cdots & y_0 \\
 y_1 & y_1 & \cdots & y_1 \\
\vdots & \vdots & \ddots & \vdots \\
 y_r & y_r & \cdots & y_r \\
 \end{pmatrix}
 \begin{pmatrix}
 td_0 & 0 & \cdots & 0 \\
 0 & td_1 & \cdots & 0 \\
 \vdots & \vdots & \ddots & \vdots \\
 0 &  0  & \cdots & td_r
 \end{pmatrix}
 \]
 lies in $\mathrm{GL}_{r+1}\bigl(R \llbracket t \rrbracket \bigr)$. This follows from the fact that 
 \[
 \mathrm{det}(M_{d_0,\ldots,d_r}) \equiv 1 \pmod t
 \]
  and therefore the determinant is a unit in $R \llbracket t \rrbracket$.
 
 Since
 \[
 (x_0, \ldots, x_r) M_{d_0, \ldots, d_r} =(x_0+t^{k+1}d_0, \ldots, x_r+t^{k+1} d_r) \smash{\rlap{,}}
 \]
 we find that there is a unimodular row $(p_0, \ldots, p_r)$ in the same $\mathrm{GL}_{r+1}$-orbit as $(x_0, \ldots, x_r)$ with the additional properties that $p_0$ is a Weierstrass polynomial of degree $k+1$ and $p_i$ is a polynomial of degree $k$ for all $i \geq 1$ (by using a suitable choice of the elements $d_i$).
 
It remains to show that $(p_0, \ldots, p_n)$ is a unimodular row over $R[t,t^{-1}]$. Since $(p_0, \ldots, p_r)$ is unimodular over $R(\!(t)\!)$, there exists a natural number $\ell \in \mathbb{N}$ and elements $z_i \in R \llbracket t \rrbracket$ such that $\sum_{i=0}^r p_i z_i=t^{\ell}$ holds. If we let $\widetilde{z_i} \in R[t]$ be the truncation of $z_i$ at degree $\ell +1$, we find that there exists a $z \in R[t]$ such that the equation
 \[
 \sum_{i=0}^r p_i \widetilde{z_i} = t^{\ell}+t^{\ell+1} z
 \]
 holds. Indeed, the left hand side is a polynomial, which has to be of the form given on the right since only terms of degree $\ell+1$ and higher are affected by the truncation of the $z_i$. Thus $t^{\ell}(1+tz)$ lies in the ideal of $R[t]$ generated by the polynomials $p_0, \ldots, p_r$. The polynomial $t^{\ell} p_0$ also lies in this ideal, so $t^{\ell}$ lies in the ideal since $1+tz$ and $p_0$ are comaximal (see Lemma~\ref{lemma:top-bottom}). Thus $(p_0, \ldots, p_r)$ is a unimodular row in the ring $R[t,t^{-1}]$ whose first entry is a Weierstrass polynomial.
\end{proof} 
 
 Our proof of Theorem~\ref{thm:Hermite_equivalence} uses two more ingredients. The first is Horrocks' Theorem, which we cite in its algebraic form.

\begin{thm}[Horrocks]\label{thm:Horrocks}
 Let $R$ be a commutative local ring and let $P$ be a finitely generated projective $R[t]$-module. If $R \langle t \rangle \ten{R[t]} P$ is free, then $P$ is free.
\end{thm}

\begin{proof}
 Horrocks' Theorem states that $P$ is free if it extends to a vector bundle on the projective line over $R$ (\cite[Theorem~1]{HORROCKS}). That this is equivalent to the algebraic form stated above is proved in \cite[Proposition~IV.2.3]{LAM}.
\end{proof} 
 
 The second ingredient of our proof of Theorem~\ref{thm:Hermite_equivalence} is the following patching technique. Let $\varphi \colon A \rightarrow B$ be a homomorphism of commutative rings and let $s \in A$. We call $\varphi$ an \emph{analytic isomorphism along $s$} if $s$ and $\varphi(s)$ are non-zero-divisors and the induced ring homomorphism
 \[
 A \slash s A \rightarrow B \slash \varphi(s) B
 \]
 is an isomorphism. In this case, the diagram
 \[
 \xymatrix{ B_{\varphi(s)}  & B \ar[l] \\ A_s \ar[u]^{\varphi_s} & A \ar[l] \ar[u]_{\varphi} }
 \]
 is called a \emph{patching diagram}, where $A_s$ and $B_{\varphi(s)}$ denote the localizations at the multiplicative sets $\{1,s,s^2, \ldots\}$ and $\{1,\varphi(s),\varphi(s)^2, \ldots\}$ respectively. The name is justified by the fact that the induced diagram
 \[
 \xymatrix{ \Mod_{B_{\varphi(s)}}^{\fgp}  & \Mod_B^{\fgp} \ar[l] \\ \Mod_{A_s}^{\fgp} \ar[u] & \Mod_{A}^{\fgp} \ar[l] \ar[u] }
 \]
 of categories of finitely projective modules is cartesian, see \cite[\S 2]{ROY}. In particular, a finitely generated projective $A_s$-module $P$ and a finitely generated projective $B$-module $Q$ can be patched to a finitely generated projective $A$-module $P^{\prime}$ if there exists an isomorphism $B_{\varphi(s)} \ten{A_s} P \cong Q_s$. The localization $P^{\prime}_s$ of $P^{\prime}$ at $s$ is then isomorphic to $P$ and the scalar extension $B \ten{A} P^{\prime}$ is isomorphic to $Q$. We will always apply this in the following form: if $P$ is a finitely generated projective $A_s$-module such that $B_{\varphi(s)} \ten{A_s} P$ is free, then there exists a finitely generated projective $A$-module $P^{\prime}$ such that $P \cong P^{\prime}_s$.

 \begin{thm}\label{thm:Laurent}
 Let $(R,\mathfrak{m})$ be a commutative local ring and let $r \geq 2$. Then the following are equivalent:
 \begin{enumerate}
 \item[(i)] All unimodular rows of length $r+1$ over $R[t]$ are completable;
 \item[(ii)] All unimodular rows of length $r+1$ over $R[t,t^{-1}]$ are completable;
 \item[(iii)] All unimodular rows of length $r+1$ over $R\langle t\rangle$ are completable;
 \item[(iv)] All unimodular rows of length $r+1$ over $R(\!(t)\!)$ are completable.
 \end{enumerate}
 \end{thm}
 
 \begin{proof}
We will prove the circle of implications $(i) \Rightarrow (iv) \Rightarrow (iii) \Rightarrow (ii) \Rightarrow (i)$.
 
 The implication $(i) \Rightarrow (iv)$. Let $(x_0, \ldots, x_r) \in \mathrm{Um}_{r+1}\bigl( R (\!(t)\!) \bigr)$ be a unimodular row. By Lemma~\ref{lemma:reduction_step}, there exists a unimodular row $(p_0,\ldots, p_r) \in \mathrm{Um}_{r+1} \bigl(R[t,t^{-1}]\bigr)$ which lies in the same $GL_{r+1}\bigl(R(\!(t)\!)\bigr)$-orbit as $(x_0,\ldots, x_n)$. Moreover, we can assume that there exists a $k \in \mathbb{N}$ such that $p_0$ is a Weierstrass polynomial and $p_i$ is a polynomial for $i \geq 1$. Since $R[t,t^{-1}] \subseteq R(\!(t)\!)$, it suffices to show that $(p_0, \ldots, p_r) \in \mathrm{Um}_{r+1} \bigl(R[t,t^{-1}]\bigr)$ is completable to a matrix in $\mathrm{GL}_{r+1} \bigl( R[t,t^{-1}]\bigr)$.
 
 We need to reduce this to the case of unimodular rows over a polynomial ring. To do this, we use a patching diagram to extend the projective module defined by $(p_0, \ldots, p_r)$ to a projective module over the polynomial ring $R[t^{-1}]$. This extension is made possible by the fact that $p_0$ is a Weierstrass polynomial. An application of Horrocks' Theorem will show that this projective module is stably free of type $1$, hence comes from a unimodular row over the polynomial ring. Such a unimodular row is by assumption completable.
 
 Here are the details. Setting $s=t^{-1}$, we can multiply each entry of $(p_0, \ldots, p_r)$ with $s^{k+1}$ by a suitable power $s^{k_i}$ of $s$ to obtain polynomials $q_0, \ldots, q_r \in R[s]$ such that $ s^{k_0} \cdot p_0=1+sq_0$ and $s^{k_i} \cdot p_i =q_i$ for $i=1,\ldots, r$. The unimodular rows
 \[
 (p_0, \ldots, p_r) \quad \text{and} \quad (1+sq_0,q_1, \ldots, q_r)
 \]
 clearly lie in the same $\mathrm{GL}_{r+1}\bigl(R[t,t^{-1}]\bigr)$-orbit. Write $P$ for the projective $R[s,s^{-1}]$-module given by the kernel of the (split) epimorphism
 \[
 (1+sq_0,q_1, \ldots, q_r) \colon R[s,s^{-1}]^{r+1} \rightarrow R[s,s^{-1}]
 \]
 of $R[s,s^{-1}]$-modules.
 
 Note that $s$ and $1+sq_0$ are comaximal. Over the localization $R[s]_{s,1+sq_0}$, the module $P_{1+sq_0}$ is clearly free (since the first entry of the unimodular row is a unit), so we can use any isomorphism
 \[
 P_{1+sq_0} \cong \bigl(R[s]^{r}_{1+sq_0}\bigr)_s
 \]
 and the (Zariski) patching diagram
 \[
 \xymatrix{R[s]_{1+sq_0,s} & R[s]_{1+sq_0} \ar[l] \\ R[s]_s \ar[u] & \ar[l] \ar[u] R[s]}
 \]
 to obtain a projective module $Q$ over $R[s]$ with $Q_s \cong P$ and $Q_{1+sq_0}$ free of rank $r$. It follows that the localization
 \[
 (Q \oplus R[s])_s \cong Q_s \oplus R[s,s^{-1}] \cong P \oplus R[s,s^{-1}] \cong R[s,s^{-1}]^{r+1}
 \]
 of $Q \oplus R[s]$ at $s$ is free of rank $r+1$. Horrocks' Theorem implies that $Q \oplus R[s]$ is free of rank $r+1$ (see Theorem~\ref{thm:Horrocks}). Thus $Q$ is the kernel of a unimodular row over the polynomial ring $R[s]$, so it follows from Assumption~$(i)$ that $Q$ is free. Thus $P \cong Q_s$ is a free $R[s,s^{-1}]$-module. This implies that the unimodular row $ (1+sq_0,q_1, \ldots, q_r)$ is completable to an invertible matrix over the ring $R[s,s^{-1}]=R[t,t^{-1}]$, which concludes the proof that $(p_0, \ldots, p_r) \in \mathrm{Um}_{r+1} \bigl( R[t,t^{-1}] \bigr)$ is completable to a matrix in $\mathrm{GL}_{r+1}\bigl(R[t,t^{-1}]\bigr)$. As already observed above, this implies that the original unimodular row $(x_0, \ldots, x_r) \in \mathrm{Um}_{r+1}\bigl( R(\!(t)\!) \bigr)$ is completable to a matrix in $\mathrm{GL}_{r+1}\bigl(R(\!(t)\!)\bigr)$, which establishes~$(iv)$.
 
  The implication $(iv) \Rightarrow (iii)$. Let $P$ be a projective $R \langle s \rangle$-module of rank $r$ such that $P \oplus R \langle s \rangle$ is free. We need to show that $P$ is free. Let $t=s^{-1}$ and recall that $R \langle s \rangle$ is equal to the localization $R[t]_{1+(t),t}$, see \cite[Proposition~IV.1.4]{LAM}. By Assumption~$(iv)$, the scalar extension of $P$ under the ring homomorphism
  \[
  R \langle s \rangle=R[t]_{1+(t),t} \rightarrow R (\!(t)\!)
  \]
  is a free module. Since $R[t]_{1+(t)} \slash (t)  \rightarrow R \llbracket t \rrbracket \slash (t)$ is an ismorphism, the homomorphism $R[t]_{1+(t)} \rightarrow R \llbracket t \rrbracket$ is an analytic isomorphism along $t$. Thus the diagram
  \[
  \xymatrix{R(\!(t)\!)=R\llbracket t \rrbracket_t & R \llbracket t \rrbracket  \ar[l] \\ 
   R \langle s \rangle=R[t]_{1+(t),t} \ar[u] & R[t]_{1+(t)} \ar[l] \ar[u] }
  \]
 is a patching diagram. It follows that the module $P$ in the lower left corner can be patched with a free module of rank $r$ in the top right corner: there exists a projective $R[t]_{1+(t)}$-module $P^{\prime}$ such that $P^{\prime}_t \cong P$ and $R \llbracket t \rrbracket \otimes P^{\prime}$ is free of rank $r$. Since the ring $R[t]_{1+(t)}$ is local (see \cite[Proposition~IV.1.4]{LAM}), the projective module $P^{\prime}$ must be free of rank $r$, so its localization $P \cong P^{\prime}_t$ is free as well. Thus the unimodular row defining $P$ is completable, which establishes~$(iii)$.
  
 The implication $(iii) \Rightarrow (ii)$. Let $P$ be a projective $R[t,t^{-1}]$-module of rank $r$ such that $P \oplus R[t,t^{-1}]$ is free and let $s=t^{-1}$. Applying~$(iii)$ to the scalar extension of $P$ in the patching diagram
 \[
 \xymatrix{ R \langle s \rangle=R[t]_{1+(t),t} & R[t]_{1+(t)} \ar[l] \\ R[t,t^{-1}] \ar[u] & R[t] \ar[l] \ar[u]}
 \]
 implies that there exists a projective $R[t]$-module $P^{\prime}$ such that $P^{\prime}_t \cong P$. Since
 \[
 (P^{\prime} \oplus R[t])_t \cong P^{\prime}_t \oplus R[t,t^{-1}] \cong P \oplus R[t,t^{-1}]
 \]
 is free, Horrocks' Theorem implies that $P^{\prime} \oplus R[t]$ is free (see Theorem~\ref{thm:Horrocks}). Applying Assumption~$(iii)$ again, we find that $R\langle t \rangle \ten{R[t]} P^{\prime}$ is free. A second application of Horrocks' Theorem thus shows that $P^{\prime}$ is free. Thus $P \cong P^{\prime}_t$ is free, as claimed.
 
 The implication $(ii) \Rightarrow (i)$. This implication is a special case of Horrocks' Theorem (see Theorem~\ref{thm:Horrocks}): if $P$ is a projective $R[t]$-module such that the localization $P_t$ is a free $R[t,t^{-1}]$-module, then $P$ is free.
 \end{proof}

 With this result in hand we can easily prove Theorem~\ref{thm:Hermite_equivalence}, which we restate for the convenience of the reader.
 
 \begin{cit}[Theorem~\ref{thm:Hermite_equivalence}]
  Let $R$ be a commutative local ring. Then the following are equivalent:
 \begin{enumerate}
 \item[(i)] The ring $R[t]$ is Hermite;
 \item[(ii)] The ring $R[t,t^{-1}]$ is Hermite;
 \item[(iii)] The ring $R \langle t \rangle$ is Hermite;
 \item[(iv)] The ring $R(\!(t)\!)$ is Hermite.
 \end{enumerate}
 \end{cit}

 \begin{proof}[Proof of Theorem~\ref{thm:Hermite_equivalence}]
 As recalled at the beginning of this section, all stably free modules over a commutative ring are free if and only if all unimodular rows are completable. Thus Theorem~\ref{thm:Hermite_equivalence} follows from Theorem~\ref{thm:Laurent}. 
 \end{proof}
 
 Theorem~\ref{thm:regular_local_equivalence} follows from Theorem~\ref{thm:Hermite_equivalence}, combined with the observation that the projective modules in question are stably free.

 \begin{cit}[Theorem~\ref{thm:regular_local_equivalence}]
 Let $R$ be a commutative regular local ring. Then the following are equivalent:
 \begin{enumerate}
 \item[(i)] All finitely generated projective $R[t]$-modules are free;
 \item[(ii)] All finitely generated projective $R[t,t^{-1}]$-modules are free;
 \item[(iii)] All finitely generated projective $R\langle t\rangle$-modules are free;
 \item[(iv)] All finitely generated projective $R(\!(t)\!)$-modules are free.
 \end{enumerate}
 \end{cit}
 
 \begin{proof}[Proof of Theorem~\ref{thm:regular_local_equivalence}]
 By Theorem~\ref{thm:Hermite_equivalence}, it suffices to prove that all finitely generated projective modules over the rings $R[t]$, $R[t,t^{-1}]$, $R \langle t \rangle$, and $R(\!(t)\!)$ are stably free. This follows from basic facts about the $K_0$-group of a localization. 
 
 Since $R$ is regular, Grothendieck's Theorem tells us that the homomorphisms
 \[
 K_0(R) \rightarrow K_0(R[t]) \rightarrow K_0(R[t,t^{-1}])
 \]
 are isomorphisms (see \cite[Theorem~XII.3.1]{BASS}). Since $R$ is local, we have $K_0(R) \cong \mathbb{Z}$. This shows that all finitely generated projective modules over the rings $R[t]$ and $R[t, t^{-1}]$ are stably free.
 
 If $A$ is a noetherian commutative ring and $S \subseteq A$ is a multiplicative set, then there exists an exact sequence
 \[
 G_0(A,S) \rightarrow G_0(A) \rightarrow G_0(A_S) \rightarrow 0
 \]
 (see \cite[Proposition~IX.6.1]{BASS}). In particular, $G_0(A) \rightarrow G_0(A_S)$ is an epimorphism. If $A$ is in addition regular, then so is $A_S$ and the Cartan homomorphisms $K_0(A) \rightarrow G_0(A)$ and $K_0(A_S) \rightarrow G_0(A_S)$ are isomorphisms (see \cite[Proposition~IX.2.1]{BASS}).
 
 By \cite[Theorem~19.5]{MATSUMURA}, both $R[t]$ and $R \llbracket t \rrbracket$ are regular since $R$ is a regular ring. Thus there exist epimorphisms
 \[
 K_0 (R[t]) \rightarrow K_0(R \langle t \rangle) \quad \text{and} \quad K_0(R \llbracket t \rrbracket) \rightarrow K_0 \bigl( R(\!(t)\!)\bigr) \smash{\rlap{.}}
 \]
 We already observed that $K_0(R[t]) \cong \mathbb{Z}$, so it follows that $K_0(R \langle t \rangle) \cong \mathbb{Z}$. Since $R \llbracket t \rrbracket$ is local, we have $K_0(R \llbracket t \rrbracket) \cong \mathbb{Z}$ and therefore $K_0 \bigl( R(\!(t)\!)\bigr) \cong \mathbb{Z}$.
 
 This shows that all finitely generated projective modules over the rings $R[t]$, $R[t,t^{-1}]$, $R \langle t \rangle$, and $R(\!(t)\!)$ are stably free. By Theorem~\ref{thm:Hermite_equivalence}, if all of these modules are free over any one of these rings, the same is true over all the other rings.
 \end{proof}
 
 The following proposition shows that the implication $(i) \Rightarrow (iv)$ of Theorem~\ref{thm:Laurent} does not generalize to rings which are not local.
 
 \begin{prop}\label{prop:Swan_example}
 There exists a commutative regular ring $R$ such that the following hold:
 \begin{enumerate}
 \item[(i)] Every stably free module over $R[t]$ is extended from $R$ (in fact, every projective module is extended);
 \item[(ii)] There exists a stably free $R(\!(t)\!)$-module which is not extended from $R$.
 \end{enumerate}
 \end{prop}
 
 \begin{proof}
 Let $R =\mathbb{C}[x_0, \ldots, x_4] \slash \sum_{i=0}^4 x_i^{2}-1$. Then $R$ is a regular ring of finite type over a field, so $(i)$ holds by Lindel's Theorem (see \cite{LINDEL}).
 
 From \cite[Theorem~2.1]{SWAN} we know that there exists a stably free $R[t,t^{-1}]$-module $P$ which does not extend to $R[t]$. We claim that the module 
 \[
 \widetilde{P} \defl R(\!(t)\!) \ten{R[t,t^{-1}]} P
 \]
  is not extended from $R$. We prove this claim by contradiction, so we assume that $\widetilde{P}$ is extended from $R$. The sequence of ring homomorphisms
 \[
 R \rightarrow R\llbracket t \rrbracket \rightarrow R\llbracket t \rrbracket_t=R(\!(t)\!)
 \]
 shows that there is in particular a projective $R \llbracket t \rrbracket$-module $Q$ such that $\widetilde{P} \cong Q_t$. From the patching diagram
 \[
 \xymatrix{R(\!(t)\!) & R \llbracket t \rrbracket \ar[l] \\ R[t,t^{-1}] \ar[u] & R[t] \ar[u] \ar[l] }
 \]
 it follows that there exists a projective $R[t]$-module $P^{\prime}$ such that $P^{\prime}_t \cong P$ and $R\llbracket t \rrbracket \ten{R[t]} P^{\prime} \cong Q$, contradicting the assertion of \cite[Theorem~2.1]{SWAN} that no such module $P^{\prime}$ can exist. Thus $\widetilde{P}$ is not extended from $R$, as claimed.
 \end{proof}
 
 A question that remains open is whether or not the implication $(i) \Rightarrow (iv)$ is also true in the case of several variables.
 
 \begin{question} Let $r,n \geq 2$ be natural numbers and let $R$ be a commutative local ring. If each unimodular row of length $r+1$ over $R[t_1, \ldots t_n]$ is completable, does it follow that each unimodular row of length $r+1$ over $R(\!(t_1, \ldots, t_n)\!)$ is completable?
 \end{question}

%% file: universal.tex
\section{Universal unimodular rows}\label{section:universal}

Throughout this section we fix a commutative ring $R$. For each $r \geq 2$, there exists a countable family of commutative $R$-algebras $A^R_{r,k}$, $k \in \mathbb{N}$ which have a ``universal'' unimodular row of length $r+1$ over the formal Laurent series ring $A^R_{r,k} (\!(t)\!)$. These rings have the property that for any unimodular row $a=(a_0,\ldots, a_n)$ over $A(\!(t)\!)$, $A$ an arbitrary commutative $R$-algebra, there exists a $k \in \mathbb{N}$ an an $R$-linear homomorphism $A_{r,k} \rightarrow A$ such that the image of the universal unimodular row is in the same $\mathrm{GL}_{r+1} \bigl( A(\!(t)\!) \bigr)$-orbit as $a$.

 We will see that the algebras $A^R_{r,k}$ are localizations of filtered colimits of $R$-algebras $\colim_n B^R_{r,k,n}$. The main result of this section is that all unimodular rows of length $r+1$ over all local $R$-algebras $A$ are completable if and only if suitable unimodular rows of length $r+1$ over $B^R_{r,k,n}$ are completable (see Theorem~\ref{thm:test_set}). It follows in particular that the commutative rings $B^{\mathbb{Z}}_{r,k,n}$ form a test set for $(\mathrm{H})$ (see Definition~\ref{dfn:test_set_for_H}). In \S \ref{section:reduction}, we will then show that the algebras $B^{\mathbb{Z}}_{r,k,n}$ yield a test set with the properties asserted in Theorem~\ref{thm:test_set_of_complete_intersections}.
 
 The proof of Theorem~\ref{thm:test_set} proceeds in several steps.
 
 \begin{enumerate}
 \item[(1)] By Theorem~\ref{thm:Hermite_equivalence}, all polynomial unimodular rows over $A$ are completable if all unimodular rows over $A(\!(t)\!)$ are completable.
 \item[(2)] By universality, it suffices to check that all unimodular rows over the formal Laurent series rings over the local rings of $A_{r,k}^R$ are completable.
 \item[(3)] By Theorem~\ref{thm:Hermite_equivalence}, this is equivalent to the assertion that all \emph{polynomial} unimodular rows over these local rings are completable.
 \item[(4)] A converse of the Quillen Patching Theorem (essentially due to Roitman) shows that suitable polynomial unimodular rows over a localization $B_S$ are completable if the same holds for $B$. Since the local rings of $A^R_{k,n}$ are localizations of $\colim_n B^R_{r,k,n}$, it suffices to check that all (suitable) polynomial unimodular rows over this filtered colimit are completable.
 \item[(5)] Since a polynomial unimodular row involves only finitely many coefficients and equations between them, all (suitable) polynomial unimodular rows over a filtered colimit are completable if this is true for all the constituents of the colimit. 
 \end{enumerate}

 Most of these steps are quite straightforward once the required $R$-algebras are defined. In Step (4), the converse of the Quillen Patching Theorem requires some additional work (see Proposition~\ref{prop:unimodular_row_and_localization}).
 
 \begin{dfn}\label{dfn:universal_rings}
  Let $r \geq 2$ be a natural number and let $S_r$ be the countable set 
  \[
  S_r \defl \bigl\{x_{i,j} \big\vert (i,j) \in \{0, \ldots, r\} \times \mathbb{N} \bigr\} \amalg \bigr\{y_{i,j} \big\vert (i,j) \in \{0, \ldots, r\} \times \mathbb{N} \bigl\} \smash{\rlap{.}}
  \]
 For $n \in \mathbb{N}$, let $p_n$ be the polynomial
 \[
 p_n \defl \sum_{i=0}^r \sum_{j=0}^n  x_{i,j} \cdot y_{i,n-j}
 \]
 in $R[S_r]$. Let $J_{k} \subseteq R[S_r]$ be the ideal generated by the set $\bigl\{ p_n \big\vert n \in \mathbb{N} \setminus \{k\} \bigr\}$. We write $B^R_{r,k} \defl R[S_r] \slash J_k$ for the quotient ring. The ring $A^R_{r,k}$ is defined as the localization
 \[
 A^R_{r,k} \defl (B^R_{r,k})_{\overline{p_k}}
 \]
 of $B^R_{r,k}$ at the image of the polynomial $p_k$ in the quotient ring.
 
 For $i=1,\ldots, r$ Let $x_i \in A_{r,k} \llbracket t \rrbracket$ be the formal power series $x_i \defl \sum_{j=0}^{\infty} \overline{x_{i,j}} t^j $ and let $y_i \in A_{r,k} \llbracket t \rrbracket$ be the formal power series $y_i \defl \sum_{j=0}^{\infty} \overline{y_{i,j}} t^j $
 \end{dfn} 
 
 \begin{lemma}
 Let $r \geq 2$ and $k \in \mathbb{N}$ be natural numbers. With the notation of Definition~\ref{dfn:universal_rings}, the two rows
 \[
 (x_0, \ldots, x_r) \quad \text{and} \quad (y_0, \ldots, y_r)
 \]
 are unimodular rows in $A^R_{r,k} (\!(t)\!)$.
 \end{lemma}
 
 \begin{proof}
 By construction, we have $\sum_{i=0}^r x_i y_i =\sum_{n=0}^\infty \overline{p_n} t^n=\overline{p_k} t^k$, which is a unit since $\overline{p_k}$ is invertible in $A^R_{r,k}$ and $t$ is invertible in $A^R_{r,k} (\!(t)\!)$.
 \end{proof}
 
 Given a homomorphism $\varphi \colon A \rightarrow B$ of commutative rings, we let
 \[
 \varphi(\!(t)\!) \colon A(\!(t)\!) \rightarrow B (\!(t)\!)
 \]
 be the ring homomorphism obtained applying $\varphi$ to each coefficient of the formal Laurent series. The following proposition shows that the unimodular rows $(x_0, \ldots, x_r)$ form a universal family.
 
 \begin{prop}\label{prop:universal_family}
 Let $A$ be a commutative $R$-algebra and let $(z_0, \ldots, z_r)$ be a unimodular row in $ \mathrm{Um}_{r+1} \bigl( A(\!(t)\!)  \bigr)$. Then there exists a $k \in \mathbb{N}$ and a homomorphism $\varphi \colon A^R_{r,k} \rightarrow A$ of $R$-algebras such that the unimodular rows
 \[
 (z_0,\ldots,z_r) \quad \text{and} \quad \bigl(\varphi(\!(t)\!)(x_0),\ldots, \varphi(\!(t)\!)(x_r) \bigr)
 \]
 lie in the same $\mathrm{GL}_{r+1}\bigl( A(\!(t)\!) \bigr)$-orbit of $\mathrm{Um}_{r+1} \bigl( A(\!(t)\!)  \bigr)$.
 \end{prop}
 
 \begin{proof}
 Let $w_i \in A(\!(t)\!)$ be elements with $\sum_{i=0}^r z_i w_i =1$. By multiplying the $z_i$ and the $w_i$ by a sufficiently high power of $t$, we can assume that $z_i \in A \llbracket t \rrbracket$, $w_i \in A \llbracket t \rrbracket$ for all $i=0, \ldots, r$ and that there exists a natural number $k \in \mathbb{N}$ such that the equation $\sum_{i=0}^r z_i w_i=t^k$ holds.
 
 If we write $z_{i,j}$ for the $j$-th coefficient of $z_i$ and $w_{i,j}$ for the $j$-th coefficient of $w_i$, we have $\sum_{i=0}^r \sum_{j=0}^n z_{i,j} w_{i,n-j}=0$ for all $n \neq k$ and $\sum_{i=0}^r \sum_{j=0}^k z_{i,j}  w_{i,k-j}=1$.
 We now use the notation introduced in Definition~\ref{dfn:universal_rings}. The first equation implies that the unique ring homomorphism $R[S_r] \rightarrow A$ sending $x_{i,j}$ to $z_{i,j}$ and $y_{i,j}$ to $w_{i,j}$ factors through the quotient ring $B^R_{r,k}$. The second equation shows that $\overline{p_k}$ is sent to a unit in $A$, which yields a unique extension $\varphi \colon A^R_{r,k} \rightarrow A$ to the localization $A^R_{r,k}$ of $B^R_{r,k}$. By construction, the ring homomorphism $\varphi(\!(t)\!)$ sends $x_i$ to $z_i$.
 \end{proof}
 
 There is an alternative construction of the rings $B^R_{r,k}$ as a filtered colimit of finitely presented commutative $R$-algebras.
 
 \begin{dfn}\label{dfn:building_blocks}
 Fix $r \geq 2$ and $k,n \in \mathbb{N}$. Let $S_{r,n}$ be the set
 \[
 S_{r,n} \defl \bigl\{x_{i,j} \big\vert (i,j) \in \{0,\ldots,r\}\times \{0,\ldots,n\}  \bigr\} \amalg \bigl\{y_{i,j} \big\vert (i,j) \in \{0,\ldots,r\}\times \{0,\ldots,n\}  \bigr\} \smash{\rlap{.}}
 \]
 For $0 \leq \ell \leq n$, let $p_{\ell}$ be the polynomial
 \[
 p_{\ell} \defl \sum_{i=0}^r \sum_{j=0}^{\ell}  x_{i,j} \cdot y_{i,\ell-j}
 \]
 in $R[S_{r,n}]$. Let $J_{k} \subseteq R[S_{r,n}]$ be the ideal generated by the set $\{p_0, \ldots, \widehat{p_k}, \ldots, p_n\}$. We let
 \[
 B^R_{r,k,n} \defl R[S_{r,n}] \slash J_k
 \]
 be the resulting quotient ring.
 \end{dfn}
 
 \begin{lemma}\label{lemma:filtered_colimit}
  There are canonical ring homomorphisms
  \[
  B^R_{r,k,n} \rightarrow B^R_{r,k,n+1} \quad \text{and} \quad B^R_{r,k,n} \rightarrow B_{r,k}^R
  \]
  which exhibit the ring $B^R_{r,k}$ of Definition~\ref{dfn:universal_rings} as filtered colimit of the chain
  \[
  \xymatrix{B^R_{r,k,0} \ar[r] & B^R_{r,k,1} \ar[r] & \ldots \ar[r] & B^R_{r,k,n} \ar[r] & \ldots}
  \]
  of commutative $R$-algebras.
 \end{lemma}
 
 \begin{proof}
 The homomorphisms are obtained by sending the elements $\overline{x_{i,j}}$ and $\overline{y_{i,j}}$ in the domain to the elements with the same name in the codomain. Since $B^R_{r,k}$ and $\mathrm{colim}_{n \in \mathbb{N}} B^R_{r,k,n}$ have the same universal property, the induced ring homomorphism $\mathrm{colim}_{n \in \mathbb{N}} B^R_{r,k,n} \rightarrow B^R_{r,k}$ is an isomorphism.
 \end{proof}
 
 For an ideal $I \subseteq A$ we write $\mathrm{Um}_{r+1}(A,I)$ for the set of unimodular rows $(a_0, \ldots, a_r) \in \mathrm{Um}_{r+1}(A)$ such that
 \[
 (a_0, \ldots, a_n) \equiv (1,0,\ldots, 0) \pmod{I}
 \]
 holds.
 
 \begin{lemma}\label{lemma:special_complement_row}
 If $(a_0, \ldots, a_r) \in \mathrm{Um}_{r+1}(A,I)$, then there exists a unimodular row $(b_0, \ldots, b_r) \in \mathrm{Um}_{r+1}(A,I)$ such that $\sum_{i=0}^r a_i b_i=1$.
 \end{lemma}
 
 \begin{proof}
 Since $(a_0, \ldots, a_r)$ is unimodular, so is $(a_0^2, \ldots, a_r^2)$. Let $b_i^{\prime}$ be elements in $A$ such that $\sum_{i=0}^r a_i^{2} b_i^{\prime}=1$ holds. If we set $b_i \defl a_i b_i^{\prime}$, then the equation $\sum_{i=0}^r a_i b_i=1$ holds by definition. Moreover, since $a_i \equiv 0 \pmod{I}$ for all $i \geq 1$, we have $b_i \equiv 0 \pmod{I}$ for all $i \geq 1$.
 
 The equations
 \[
 1=\sum_{i=0}^r a_i b_i \equiv a_0 b_0 \pmod{I}
 \]
 and $a_0 \equiv 1 \pmod{I}$ show that $b_0 \equiv 1 \pmod{I}$ holds. Finally, $(b_0, \ldots, b_r)$ is a unimodular row since $\sum_{i=0}^r a_i b_i=1$.
 \end{proof}
 
 The following proposition is a variation of a result of Roitman (see \cite[Proposition~2]{ROITMAN}).
 
 \begin{prop}\label{prop:unimodular_row_and_localization}
 Let $A$ be a commutative ring and let $S \subseteq A$ be a multiplicative subset. Let $r \geq 2$. If each unimodular row  $\bigl(a_0(t), \ldots, a_r(t) \bigr) \in \mathrm{Um}_{r+1}\bigl(A[t],(t) \bigr)$ is completable, then each unimodular row $\bigl(b_0(t), \ldots, b_r(t)\bigr)\mathrm{Um}_{r+1}\bigl(A_S[t],(t) \bigr)$ is completable.
 \end{prop}
 
 \begin{proof}
 By Lemma~\ref{lemma:special_complement_row}, we can choose $\bigl(c_0(t), \ldots, c_r(t) \bigr) \in \mathrm{Um}_{r+1}\bigl(A_S[t],(t) \bigr)$ with $\sum_{i=0}^r b_i(t)c_i(t)=1$. Since the polynomials $b_i(t)$ and $c_j(t)$ have a finite number of non-zero coefficients, there exists an element $s \in S$ such that all these coefficients can be written as fractions $x \slash s^{\ell}$ for some $x \in A$ and $\ell \in \mathbb{N}$. Thus for $k \in \mathbb{N}$ sufficiently large, we have $b_i(s^k t) \in A[t]$ and $c_j(s^k t) \in A[t]$. Since $\sum_{i=0}^r b_i(s^k t)c_i(s^k t)=1$, the polynomials $a_i(t) \defl b_i(s^k t)$ define a unimodular row in $\mathrm{Um}_{r+1}\bigl(A[t],(t)\bigr)$ (note that $a_i(0)=b_i(0)$, which shows that the row has the prescribed form modulo the ideal $(t)$.)
 
 By assumption, there exists a matrix $A=\bigl(a_{ij}(t)\bigr) \in \mathrm{GL}_{r+1}\bigl(A[t]\bigr)$ with first row $a_{1,i}(t)=a_{i-1}(t)$, $i=1, \ldots, r+1$. If we consider $A$ as an element of $\mathrm{GL}_{r+1}\bigl(A_S[t]\bigr)$ and apply the automorphism $\varphi \colon A_S[t] \rightarrow A_S[t]$ which fixes $A_S$ and sends $t$ to $s^{-k} t$, we get the desired matrix $B=\varphi(A) \in \mathrm{GL}_{r+1}\bigl(A_S[t]\bigr)$ with first row  equal to $\bigl(b_0(t), \ldots, b_r(t) \bigr)$.
 \end{proof}
 
 With this in hand, we are ready to prove the main result of this section. 
 
 \begin{thm}\label{thm:test_set}
 Let $R$ be a commutative ring, $r \geq 2$ a natural number. The following are equivalent:
 \begin{enumerate}
 \item[(i)] For all commutative local $R$-algebras $A$, all unimodular rows in $\mathrm{Um}_{r+1}(A[t])$ are completable;
 \item[(ii)] For all commutative $R$-algebras $A$, all unimodular rows in $\mathrm{Um}_{r+1}\bigl(A[t],(t)\bigr)$ are completable;
 \item[(iii)] For all $k, n \in \mathbb{N}$, all unimodular rows in $\mathrm{Um}_{r+1}\bigl(B_{r,k,n}^R[t],(t)\bigr)$ are completable, where $B_{r,k,n}^R$ denotes the $R$-algebra of Definition~\ref{dfn:building_blocks}.
 \end{enumerate}
 \end{thm}
 
 \begin{proof}
 The implication $(i) \Rightarrow (ii)$. Let $A$ be a commutative $R$-algebra and let $\alpha=(\alpha_0, \ldots, \alpha_r) \in \mathrm{Um}_{r+1}\bigl(A[t],(t)\bigr)$. Let $P$ be the kernel of $\alpha \colon A[t]^{r+1} \rightarrow A[t]$. Since $P \slash tP$ is free, it suffices to show that $P$ is extended from $A$. By Quillen's Patching Theorem (more precisely by its corollary \cite[Corollary~V.1.8]{LAM}), we only need to check that $P_{\mathfrak{m}}$ is a free $A_{\mathfrak{m}}[t]$-module for every maximal ideal $\mathfrak{m}$ of $A$. This follows from Assumption~$(i)$.
 
 The implication $(ii) \Rightarrow (i)$. For a local ring $A$, every $\mathrm{GL}_{r+1}\bigl(A[t]\bigr)$-orbit of $\mathrm{Um}_{r+1}\bigl(A[t]\bigr)$ contains an element of $\mathrm{Um}_{r+1}\bigl(A[t],(t)\bigr)$. Indeed, for any unimodular row $\alpha(t) = \bigl(\alpha_0(t), \ldots, \alpha_r(t)\bigr)$ over $A[t]$ we can find a matrix $M \in \mathrm{GL}_{r+1}(A)$ such that $\bigl(\alpha_0(0),\ldots, \alpha_r(0) \bigr)M=(1,0,\ldots, 0)$ (since $A$ is local). Then $\alpha^{\prime} \defl \alpha(t) M$ lies in $\mathrm{Um}_{r+1}\bigl(A[t],(t)\bigr)$, hence it is completable by Assumption~$(ii)$.
 
 The implication $(ii) \Rightarrow (iii)$ is clear, so it remains to show $(iii) \Rightarrow (i)$. Let $(A,\mathfrak{m})$ be an arbitrary commutative local $R$-algebra. We need to show that all unimodular rows in $\mathrm{Um}_{r+1}\bigl(A[t]\bigr)$ are completable. By the implication $(iv) \Rightarrow (i)$ of Theorem~\ref{thm:Laurent}, it suffices to prove the assertion that all unimodular rows $\alpha=(a_0, \ldots, a_r) \in \mathrm{Um}_{r+1}\bigl(A(\!(t)\!)\bigr)$ are completable.
 
 By Proposition~\ref{prop:universal_family}, there exists a $k \in \mathbb{N}$ and a homomorphism of $R$-algebras $\varphi \colon A^R_{r,k} \rightarrow A$ such that $\alpha$ is in the same $\mathrm{GL}_{r+1}$-orbit as the image of some $\beta \in \mathrm{Um}_{r+1}\bigl(A^R_{r,k}(\!(t)\!)\bigr)$ under $\varphi(\!(t)\!) \colon A^R_{r,k}(\!(t)\!) \rightarrow A(\!(t)\!)$. Without loss of generality, we can assume that  $\alpha$ is equal to the image of $\beta$.
 
 Let $\mathfrak{p} \defl \varphi^{-1}(\mathfrak{m})$ and let
 \[
 \xymatrix{ (A^R_{r,k})_{\mathfrak{p}}  \ar[r]^-{\varphi^{\prime}} & A \\ A^R_{r,k} \ar[ru]_{\varphi} \ar[u] }
 \]
 be the canonical factorization through the localization. This factorization immediately tells us that $\alpha$ also is the image of some $\beta^{\prime} \in \mathrm{Um}_{r+1}\bigl( (A^R_{r,k})_{\mathfrak{p}}(\!(t)\!)\bigr)$ under the homomorphism $\varphi^{\prime}(\!(t)\!) \colon (A^R_{r,k})_{\mathfrak{p}}(\!(t)\!) \rightarrow A(\!(t)\!)$, so it suffices to show that all unimodular rows in $\mathrm{Um}_{r+1}\bigl( (A^R_{r,k})_{\mathfrak{p}}(\!(t)\!)\bigr)$ are completable.
 
 By the implication $(i) \Rightarrow (iv)$ of Theorem~\ref{thm:Laurent}, it suffices to prove the assertion that all unimodular rows $\gamma \in \mathrm{Um}_{r+1}\bigl( (A^R_{r,k})_{\mathfrak{p}}[t]\bigr)$ are completable. Since $(A^R_{r,k})_{\mathfrak{p}}$ is local, there exists a matrix $M \in \mathrm{GL}_{r+1}\bigl((A^R_{r,k})_{\mathfrak{p}}\bigr)$ such that $\gamma(0)M=(1, 0,\ldots, 0)$, so $\gamma^{\prime} \defl \gamma M$ lies in $\mathrm{Um}_{r+1} \bigl( (A^R_{r,k})_{\mathfrak{p}}[t],(t)\bigr)$. Therefore it suffices to show that all unimodular rows in the latter set are completable.
 
 By Proposition~\ref{prop:unimodular_row_and_localization}, this is the case if all unimodular rows in $\mathrm{Um}_{r+1} \bigl( A^R_{r,k}[t],(t)\bigr)$ are completable. Since $A^R_{r,k}$ is the localization of $B^R_{r,k}$ at $\overline{p_k}$ (see Definition~\ref{dfn:universal_rings}), we can apply Proposition~\ref{prop:unimodular_row_and_localization} once more to reduce the claim to the assertion that all unimodular rows in $\mathrm{Um}_{r+1}\bigl(B^R_{r,k}[t],(t)\bigr)$ are completable.
 
 Recall that $B^R_{r,k}$ is the directed colimit of the $R$-algebras
 \[
   \xymatrix{B^R_{r,k,0} \ar[r] & B^R_{r,k,1} \ar[r] & \ldots \ar[r] & B^R_{r,k,n} \ar[r] & \ldots}
 \]
 of Definition~\ref{dfn:building_blocks} (see Lemma~\ref{lemma:filtered_colimit}). Since every polynomial unimodular row $\delta \in \mathrm{Um}_{r+1}\bigl(B^R_{r,k}[t],(t)\bigr)$ is defined using only a finite number of coefficients and equations, there exists an $n \in \mathbb{N}$ such that $\delta$ lies in the image of
 \[
 \mathrm{Um}_{r+1} \bigl( B^R_{r,k,n}[t],(t) \bigr) \rightarrow \mathrm{Um}_{r+1}\bigl(B^R_{r,k}[t],(t)\bigr) \smash{\rlap{.}}
 \]
  Thus $\delta$ is completable by Assumption~$(iii)$.
 \end{proof}

%% file: reduction.tex
\section{A test set of unique factorization domains}\label{section:reduction}

The aim of this section is to show that the algebras $B^R_{r,k,n}$ of Definition~\ref{dfn:building_blocks} inherit certain nice properties of $R$. The argument uses an iterative procedure involving auxiliary rings. It relies on two basic facts: if $B$ is a unique factorization domain and $f \in B$ is irreducible, then $f$ is prime, so the quotient ring $B \slash f$ is a domain. Secondly, we use Nagata's Criterion for a ring to be a unique factorization domain: if $A$ is a noetherian domain, $a \in A$ a prime element, and the localization $A_a$ is a unique factorization domain, then $A$ is a unique factorization domain.

 In order to show that $a \in A$ is prime, we will use an isomorphism $A \slash a \cong B \slash f$ for a suitable unique factorization domain $B$ and irreducible element $f \in B$. In order to show the irreducibility of this element $f$, we use the following lemmas as a base case and as a reduction to this base case respectively.
 
 \begin{lemma}\label{lemma:xy+zw_irreducible}
 Let $R$ be an integral domain. Then the element $xy+zw$ of the polynomial ring $R[x,y,z,w]$ is irreducible.
 \end{lemma}

\begin{proof}
 Assume $xy+zw=pq$. Since $R$ is a domain, both $p$ and $q$ must be homogeneous. By contemplating the product of two generic homogeneous polynomials of degree $1$ and by using the fact that the monomials form an $R$-basis of $R[x,y,z,w]$ one finds that $p$ and $q$ cannot both be of degree $1$. Thus one of them has to be of degree $0$, that is, an element of $R$. Without loss of generality we have $p \in R$. If $p$ is not a unit, then there exists a maximal ideal $\mathfrak{m} \subseteq R$ with $p \in \mathfrak{m}$. It follows that $xy+zw=0$ in $R\slash \mathfrak{m}[x,y,z,w]$, a contradiction. Thus $p$ is a unit, so $xy+zw$ is irreducible.
\end{proof}

 Given a polynomial ring on some set of variables $S$, a \emph{variable reduction} is a set $T$ and a function $\varphi \colon S \rightarrow T \cup \{0\}$. For any commutative ring $R$, it induces an $R$-algebra homomorphism
 \[
 \varphi_{\ast} \colon R[S] \rightarrow R[T]
 \]
 sending each $s \in S$ to $\varphi(s) \in R[T]$.
 
 Since each variable of $S$ is sent either to a single variable in $T$ or to $0$, we find that $\varphi_{\ast}$ sends any homogeneous polynomial of degree $d$ either to a homogeneous polynomial of degree $d$ or to $0$.
 
 The following lemma will be used to show that an element $\overline{g}$ in some quotient ring $R[S] \slash f_1, \ldots, f_m$ is irreducible. It is proved in three steps. First a product decomposition of $\overline{g}$ and some properties of the quotient ring are used to establish an equation in $R[S]$ involving $g$ and all the $f_i$. In the second step, we assume that certain variable reductions exist which send almost all the $f_i$ to $0$. The equation of step one then implies that one of the factors of $\overline{g}$ lies in $R$. In the final step, we use an assumption on the monomials occurring in $g$ and the $f_i$ to ensure that this factor is a unit.

\begin{lemma}\label{lemma:irreducible_criterion}
 Let $R$ be a unique factorization domain and let $S$ be a set. Let $f_1, \ldots, f_m \in R[S]$ be homogeneous polynomials of degree $2$. Assume that the quotient ring $B \defl R[S] \slash f_1, \ldots, f_m$ is a domain. Finally, assume that there exists a monomial whose coefficient in $g$ is a unit and whose coefficient in $f_i$ is $0$ for all $i=1, \ldots, m$. Then $\overline{g} \in B$ is irreducible in each of the following cases:
 \begin{enumerate}
 \item[(1)] There exists a variable reduction $\varphi \colon S \rightarrow \{c_0,d_0,e_0,f_0\} \cup \{0\}$ such that $\varphi_{\ast} g=c_0d_0+e_0 f_0$ and $\varphi_{\ast} f_i=0$ for all $i=1,\ldots, m$;
 \item[(2)] There exists a variable reduction $\varphi \colon S \rightarrow \{c_0,d_0,e_0,f_0,c_{\ell},d_{\ell},e_{\ell},f_{\ell} \} \cup \{0\}$ such that $\varphi_{\ast} g=c_0d_{\ell}+c_{\ell} d_0+e_0 f_{\ell}+e_{\ell} f_0$ and $\varphi_{\ast} f_i=0$ for all $i=1,\ldots, m$;
 \item[(3)] There exists a variable reduction $\varphi \colon S \rightarrow \{c_0,d_0,e_0,f_0,c_{\ell},d_{\ell},e_{\ell},f_{\ell} \} \cup \{0\}$ and a $j \in \{1, \ldots, m\}$ such that $\varphi_{\ast} g=c_0d_{\ell}+c_{\ell} d_0+e_0 f_{\ell}+e_{\ell} f_0$, the image $\varphi_{\ast} f_i$ is $0$ for all $i \neq j$, and $\varphi_{\ast} f_j=c_{\ell} d_{\ell} + e_{\ell} f_{\ell}$.
 \end{enumerate}
\end{lemma}

 \begin{proof}
 Since all the polynomials $f_i$ are homogeneous, the quotient ring $B$ is $\mathbb{N}$-graded, with degree $d$ part given by the image of the homogeneous polynomials of degree $d$. Since $B$ is a domain it follows that $\overline{g}$ is not a unit (since it has degree $>0$). The fact that $B$ is a domain also implies that the elements $\overline{p}$, $\overline{q}$ occurring in a factorization $\overline{g}=\overline{p} \overline{q}$ must be homogeneous of degree at most $2$, and the sum of these degrees must be $2$. By definition of the grading we can assume that the preimages $p$, $q$ of $\overline{p}$, $\overline{q}$ in $R[S]$ are homogeneous polynomials of the same degree. It follows that there exist polynomials $\beta_i \in R[S]$ such that the equation
 \[
 g=pq+\textstyle\sum_{i=1}^m \beta_i f_i
 \]
 holds in $R[S]$. Writing $\beta_i=\gamma_i + \delta_i$ where $\gamma_i \in R$ is the constant part of $\beta_i$, we thus have $ g=pq+\sum_{i=1}^m \gamma_i f_i + \sum_{i=1}^m \delta_i f_i$. No monomial of degree greater than $2$ occurs in $g -pq -\sum_{i=1}^m \gamma_i f_i$ and all monomials of $\sum_{i=1}^m \delta_i f_i$ have degree $\geq 3$, so we must have $\sum_{i=1}^m \delta_i f_i=0$. We thus have
 \begin{equation}\label{eqn:reducible}
 g=pq+\textstyle \sum_{i+1}^m \gamma_i f_i \tag{$\ast$}
 \end{equation}
 in $R[S]$ for some $\gamma_i \in R$, $p$, $q$ homogeneous of degree $\leq 2$, and the sum of the degrees of $p$ and $q$ is $2$. 
 
 We claim that the degrees of $p$ and $q$ cannot both be equal to $1$ if variable reductions as in Case~(1)--(3) exist. We prove this by contradiction, so we assume that both $p$ and $q$ have degree $1$.
 
 Case~(1): The variable reduction in Case~(1), combined with Equation~\eqref{eqn:reducible}, implies that
 \[
 c_0 d_0 +e_0 f_0 =\varphi_{\ast} g=\varphi_{\ast} p \cdot \varphi_{\ast} q
 \]
 is a product of homogeneous polynomials of degree $1$. This contradicts the irreducibility of $c_0 d_0 + e_0 f_0 \in R[c_0,d_0,e_0,f_0]$ (see Lemma~\ref{lemma:xy+zw_irreducible}).
 
 Case~(2): From the variable reduction in Case~(2) we can construct one as in Case~(1) by sending $e_0$, $e_{\ell}$, $f_0$, and $f_{\ell}$ to $0$, leading to the same contradiction as in Case~(1).
 
 Case~(3): The variable reduction in Case~(3), applied to Equation~\eqref{eqn:reducible} yields the equation
 \begin{equation}\label{eqn:case_3}
 c_0 d_{\ell}+c_{\ell} d_0+e_0 f_{\ell}+e_{\ell} f_0 = \varphi_{\ast}(p) \cdot \varphi_{\ast}(q)+\gamma_j (c_{\ell} d_{\ell}+e_{\ell} f_{\ell}) \tag{$\ast \ast$}
 \end{equation}
 in $R[c_0, c_{\ell},d_0, d_{\ell}, e_{0}, e_{\ell}, f_0, f_{\ell}]$. Our assumption implies that both $\varphi_{\ast}(p)$ and $\varphi_{\ast}(q)$ are homogeneous of degree $1$.
 
 If we send each $c_0, d_0, e_0, f_0$ to $0$, we find that the equation
 \[
 \gamma_j(c_{\ell} d_{\ell}+e_{\ell} f_{\ell})=-\overline{\varphi_{\ast} (p)} \cdot \overline{\varphi_{\ast} (q)}
 \]
 holds in $R[c_{\ell},d_{\ell}, e_{\ell}, f_{\ell}]$. Note that here we use the fact that $\gamma_j \in R$ so that $\overline{\gamma_j}=\gamma_j$. Since $R$ is a unique factorization domain and $c_{\ell} d_{\ell}+ e_{\ell} f_{\ell}$ is irreducible (see Lemma~\ref{lemma:xy+zw_irreducible}), the element $c_{\ell} d_{\ell}+ e_{\ell} f_{\ell}$ is prime. If $\gamma_j \neq 0$, then this prime element divides either $\overline{\varphi_{\ast} (p)}$ or $\overline{\varphi_{\ast} (q)}$, contradicting the fact that these are both homogeneous of degree $1$. Thus we must have $\gamma_j=0$.
 
 Combining this fact with Equation~\eqref{eqn:case_3}, we find that the equation
 \[
 c_0 d_{\ell}+c_{\ell} d_0+e_0 f_{\ell}+e_{\ell} f_0 = \varphi_{\ast}(p) \cdot \varphi_{\ast}(q)
 \]
 holds. If we send $e_0$, $e_\ell$, $f_0$, and $f_{\ell}$ to $0$, we find that $c_0 d_{\ell}+c_{\ell} d_0$ is a product of two homogeneous polynomials of degree $1$, contradicting the fact that it is irreducible in $R[c_0, c_{\ell}, d_0, d_{\ell}]$ (see Lemma~\ref{lemma:xy+zw_irreducible}).
 
 This case-by-case analysis shows that the degrees of $p$ and $q$ cannot both be equal to $1$. Thus we can assume without loss of generality that $p \in R$ and the degree of $q$ is $2$. We claim that $p$ is a unit in $R$, which will establish the fact that $\overline{g}$ is irreducible in $B$. If $p$ is not a unit, then there exists a maximal ideal $\mathfrak{m} \subseteq R$ with $p \in \mathfrak{m}$. It follows from Equation~\eqref{eqn:reducible} that
 \[
 \overline{g}=\textstyle\sum_{i=1}^m \overline{\gamma_i} \overline{f_i}
 \]
 holds in $R \slash \mathfrak{m} [S]$. But by assumption, one monomial of $\overline{g}$ has a unit as coefficient and its coefficient in each $\overline{f_i}$ is $0$, so the above equation contradicts the fact that the monomials form an $R \slash \mathfrak{m}$-basis of $R\slash \mathfrak{m} [S]$. Thus $p$ is indeed a unit, which shows that $\overline{g} \in B$ is irreducible.
 \end{proof}
 
 In order to apply Nagata's Criterion, we will need to show that certain localizations are unique factorization domains. In all cases, we will use the following lemma.
 
 \begin{lemma}\label{lemma:localization_isomorphism}
 Let $A$ be a commutative ring, $f_0, \ldots, f_{\ell} \in A[t_0, \ldots, t_{\ell}]$, and let $B \defl A[t_0, \ldots, t_{\ell}] \slash f_0, \ldots, f_{\ell}$. Assume that there exist $a, g_0 \in A$ and $g_i \in A[t_0, \ldots, t_{i-1}]$ for $0 < i \leq \ell$ such that the equations
 \[
 f_i=at_i - g_i(t_0, \ldots, t_{i-1})
 \]
 hold for $i=0, \ldots, \ell$. Then the canonical homomorphism $\varphi \colon A \rightarrow B$ induces an isomorphism of the localizations $\varphi_a \colon A_a \cong B_a$ at the multiplicative set $\{1,a,a^2,\ldots\}$.
 \end{lemma}
 
 \begin{proof}
 We construct an $A$-homomorphism $\psi \colon B_a \rightarrow A_a$ as follows. Let $r_0 \defl a^{-1} g_0 \in A_a$ and inductively set $r_i \defl a^{-1} g(r_0, \ldots, r_{i-1}) \in A_a$ for $i >0$. Sending $t_i$ to $r_i$ defines an $A$-homomorphism
 \[
 \psi^{\prime} \colon A[t_0, \ldots, t_{\ell}] \rightarrow A_a \smash{\rlap{.}}
 \] 
 We claim that the $f_i$ lie in the kernel of $\psi^{\prime}$. Indeed, we have $\psi^{\prime}(f_i)=ar_i-g(r_0,\ldots,r_{i-1})$, which is $0$ by construction of the $r_i$. Thus $\psi^{\prime}$ uniquely factors through the quotient $B$, yielding an $A$-homomorphism $B \rightarrow A_a$. Since $a$ is sent to a unit, this factors uniquely through the localization, and this yields the desired homomorphism $\psi \colon B_a \rightarrow A_a$.
 
 By construction, we have $\psi \circ \varphi_a = \id_{A_a}$. To see that $\varphi_a \circ \psi$ is the identity, we only need to check that this composite sends $\overline{t_i}$ to $\overline{t_i}$ for $0 \leq i \leq \ell$. This boils down to checking that $\overline{t_i}=r_i$ holds in $B_a$. Since $a$ is a unit, this is the case if and only if $a\overline{t_i}=ar_i$ holds. We have $a r_0=g_0$, so this holds for $i=0$ by definition of $f_0$. By iterating we find that
 \[
 a r_i =g_i(r_0, \ldots, r_i)=g_i(\overline{t_0}, \ldots, \overline{t_{i-1}})
 \]
 holds, and thus $a r_i=a \overline{t_i}$ by definition of $f_i$. Since $a$ is invertible in $B_a$, we have $\overline{t_i}=r_i$, as claimed.
 \end{proof}
 
 Let $R$ be a commutative ring. Recall from Definition~\ref{dfn:building_blocks} that $S_{r,n}$ is the set
 \[
 S_{r,n} = \bigl\{x_{i,j} \big\vert (i,j) \in \{0,\ldots,r\}\times \{0,\ldots,n\}  \bigr\} \amalg \bigl\{y_{i,j} \big\vert (i,j) \in \{0,\ldots,r\}\times \{0,\ldots,n\}  \bigr\} \smash{\rlap{,}}
 \]
 $p_{\ell}$ is the polynomial
 \[
 p_{\ell} = \sum_{i=0}^r \sum_{j=0}^{\ell}  x_{i,j} \cdot y_{i,\ell-j}
 \]
 in $R[S_{r,n}]$, and $J_{k} \subseteq R[S_{r,n}]$ is the ideal generated by the set $\{p_0, \ldots, \widehat{p_k}, \ldots, p_n\}$. The commutative $R$-algebra
 \[
 B^R_{r,k,n} = R[S_{r,n}] \slash J_k
 \]
 is the resulting quotient algebra of the polynomial $R$-algebra $R[S_{r,n}]$. Throughout the next theorem, we keep $r \geq 2$, $k$, and $n$ fixed and we let $S \defl S_{r,n}$.
 
 \begin{thm}\label{thm:building_blocks_ufd}
 Let $R$ be a noetherian unique factorization domain. Then the sequence
 \[
 p_n, p_{n-1}, \ldots, \widehat{p_k}, \ldots, p_0
 \]
 is a regular sequence in the polynomial algebra $R[S]$. The quotient algebra $B^{R}_{r,k,n}$ of Definition~\ref{dfn:building_blocks} is a noetherian unique factorization domain.
 \end{thm}
 
\begin{proof}
 We use the notation $a_j \defl x_{0,j}$ and $b_j\defl y_{0,j}$ for $j=0,\ldots, n$. For $0 \leq \ell \leq n$ and $0 \leq i \leq \ell$ we let
 \[
 C_{\ell,i} \defl R[S] \slash p_n, p_{n-1}, \ldots, \widehat{p_k}, \ldots, p_{\ell}, a_0, \ldots, a_i
 \]
 and we let $C_{\ell, -1} \defl R[S] \slash p_n, p_{n-1}, \ldots, \widehat{p_k}, \ldots, p_{\ell}$. The final sequence of auxiliary algebras we need is $C_{n+1,-1} \defl R[S]$ and $C_{n+1,i} \defl R[S] \slash a_0, \ldots, a_i$ for $0 \leq i \leq n$. These $R$-algebras fit in the matrix
 \[
 \begin{matrix}
 C_{n+1,-1} & C_{n+1,0} && \cdots && C_{n+1,n} \\
 C_{n,-1} & C_{n,0} && \cdots && C_{n,n} \\
 \vdots & \vdots && & \iddots \\
 C_{\ell, -1} & C_{\ell, 0}  & \cdots & C_{\ell, \ell} \\
 \vdots & \vdots & \iddots \\
 C_{0,-1} & C_{0,0}
 \end{matrix}
 \]
 where each algebra is the quotient of the one above it by some $\overline{p_{\ell}}$ and a quotient of the one to the left by  some $a_i$. We claim that all these algebras $C_{\ell,i}$ are domains and that $C_{\ell, i}$ is a unique factorization domain if $i < \ell$ (that is, if it is above the diagonal in the diagram above).
 
 We first give an outline how this claim is proved. We assume by induction that the claim holds for all the rows above $\ell$. We will use Lemma~\ref{lemma:irreducible_criterion} to show that $\overline{p_{\ell}} \in C_{\ell+1,i}$ is irreducible. Since $C_{\ell, i} \cong C_{\ell+1,i} \slash \overline{p_{\ell}}$, this shows that the $\ell$-th row of the diagram consists of domains.
 
 Since $C_{\ell,i} \cong C_{\ell,i-1} \slash a_i$ is a domain, we find that $a_i \in C_{\ell,i-1}$ is a prime element. We will use Lemma~\ref{lemma:localization_isomorphism} to show that the localization $(C_{\ell,i-1})_{a_i}$ is a unique factorization domain. Since $C_{\ell,i-1}$ is noetherian, Nagata's Criterion implies that $C_{\ell,i-1}$ is a unique factorization domain.
 
 We now give a detailed proof of the claim that each $C_{\ell, i}$ is a domain and that $C_{\ell, i}$ is a unique factorization domain if $i < \ell$. We prove this by downward induction on $\ell$. If $\ell = n+1$, then each $C_{n+1,i}$ is a polynomial $R$-algebra for $0 \leq i \leq n$, hence a unique factorization domain.
 
 By induction we can assume that $C_{\ell+1, i}$ is a unique factorization domain for all $-1 \leq i \leq \ell $. We first need to establish that each $C_{\ell,i}$ is a domain for $0 \leq i \leq \ell$. If $\ell=k$, then $C_{\ell,i}=C_{\ell+1, i}$ and there is nothing to do. If $\ell \neq k$, then we have isomorphisms $C_{\ell, i} \cong C_{\ell+1,i} \slash \overline{p_{\ell}}$, so it suffices to check that $\overline{p_{\ell}}$ is irreducible in the unique factorization domain $C_{\ell+1,i}$.
 
 All monomials of $p_{\ell}$ have coefficient $1 \in R^{\times}$, and all the monomials occurring in $p_{\ell}$ have coefficient $0$ in $p_{m}$ for $m > \ell$. The algebra $C_{\ell+1,i}$ can be seen as a quotient of the polynomial algebra $R[S \setminus \{a_0, \ldots, a_i\} ]$. Since the quotient is a domain, we can apply Lemma~\ref{lemma:irreducible_criterion}, so we need to show that there are variable reductions as in Case~(1)--(3) of the Lemma~\ref{lemma:irreducible_criterion}.
 
 If $\ell=0$, then we use the variable reduction $\varphi$ given by $x_{1,0} \mapsto c_0$, $y_{1,0} \mapsto d_0$, $x_{2,0}\mapsto e_0$, $y_{2,0} \mapsto f_0$ and $s \mapsto 0$ for all other $s \in S \setminus \{a_0, \ldots, a_{i}\}$. Then $\varphi_{\ast}(p_{\ell})=c_0 d_0+e_0 f_0$ and $\varphi_{\ast}(p_m)=0 $ for all $m >0$, so we are in Case~(1) of Lemma~\ref{lemma:irreducible_criterion}.
 
 If $\ell \neq 0$, we use the variable reduction $\varphi$ given by
 \[
 \begin{matrix}
 x_{1,0} & x_{1,\ell} & y_{1,0} & y_{1,\ell} &  x_{2,0} & x_{2,\ell} & y_{2,0} & y_{2,\ell} \\
 \downmapsto & \downmapsto & \downmapsto & \downmapsto & \downmapsto & \downmapsto & \downmapsto & \downmapsto \\
 c_0 & c_{\ell} & d_0 & d_{\ell} & e_0 & e_{\ell} & f_0 & f_{\ell}
 \end{matrix}
 \]
 and $\varphi(s)=0$ for all other $s \in S \setminus \{a_0, \ldots, a_i\}$. Then we have
 \[
 \varphi_{\ast}(p_{\ell}) = c_0 d_{\ell}+ c_{\ell} d_0 + e_0 f_{\ell} + e_{\ell} f_0
 \]
 and $\varphi_{\ast}(p_m)=0$ for all $m \neq \ell, 2\ell$. If $2 \ell = k$ or $2 \ell > n$, then we have a variable reduction as in Case~(2) of Lemma~\ref{lemma:irreducible_criterion}. If not, then $\varphi_{\ast}(p_{2 \ell})=c_{\ell} d_{\ell} + e_{\ell} f_{\ell}$ and we are in Case~(3) of Lemma~\ref{lemma:irreducible_criterion}. Thus $\overline{p_{\ell}}$ is irreducible in the unique factorization domain $C_{\ell+1, i}$, which establishes the fact that $C_{\ell,i} \cong C_{\ell+1,i} \slash \overline{p_\ell}$ is a domain for $-1 \leq i \leq \ell$.
 
 It remains to show that $C_{\ell,i}$ is a unique factorization domain for all $-1 \leq i < \ell$. Since $C_{\ell, i-1} \slash a_i \cong C_{\ell,i}$ is a domain for $i=0, \ldots, \ell$, we know that $a_i \in C_{\ell, i-1}$ is a prime element. The $R$-algebra $C_{\ell, i-1}$ is noetherian since it is finitely presented. By Nagata's Criterion (see \cite[Theorem~20.2]{MATSUMURA}), it suffices to check that the localization $(C_{\ell,i-1})_{a_i}$ is a unique factorization domain for all $i=0,\ldots, \ell$. We will use Lemma~\ref{lemma:localization_isomorphism} to show that each one of these rings is a localization of some polynomial ring $R[T]$, where $T$ is a subset of $S$.
 
 To see that this lemma is applicable, it is instructive to consider a low-dimensional example. For $n=6$, the polynomials $p_m$ are depicted in the table
 \[
 \arraycolsep=1.4pt
 \begin{array}{cc|cccccccccccccccc}
p_6 &&& a_0 b_6 & + & a_1 b_5 & + & a_2 b_4  & + & a_3 b_3  & + & a_4 b_2  & +& a_5b_1 & + &  a_6 b_0 & + & \ldots \\
p_5 &&& a_0 b_5 & + & a_1 b_4 & + & a_2 b_3 & + & a_3 b_2 & + & a_4 b_1 &+& a_5 b_0 & + & \ldots \\
p_4 &&& a_0 b_4 & + & a_1 b_3 & + & a_2 b_2 & + & a_3 b_1 & + & a_4 b_0 & + & \ldots \\
p_3 &&& a_0 b_3 & + & a_1 b_2 & + & a_2 b_1 & + & a_3 b_0 & + & \ldots \\
p_2 &&& a_0 b_2 & + & a_1 b_1 & + & a_2 b_0 & + & \ldots \\
p_1 &&& a_0 b_1 & + & a_1 b_0 & + & \ldots \\
p_0 &&& a_0 b_0 & + & \ldots
 \end{array}
 \]
 where the omitted parts do not involve any of the variables $a_i$ and $b_j$. If $0 \leq i \leq \ell$, then the polynomials defining $C_{\ell,i-1}$ are obtained from the ones above by deleting the first $i-1$ columns. If $k$ lies between $0$ and $6$, we also have to remove the respective row. The key observation is that the variable $b_j$ with the largest index in a given row always occurs in the first column in the monomial $a_i b_j$, and it does not occur in any other monomial of that row or in any of the rows below it. Thus we can set the variables $t_m$ to be equal to these variables $b_j$ to see that the conditions of Lemma~\ref{lemma:localization_isomorphism} are satisfied for the ring $R[T]$, where $T$ denotes the set of the remaining variables.
 
 If, for example, $k=3$ and $i=2$, $\ell=2$, then the polynomials are
\[
 \arraycolsep=1.4pt
 \begin{array}{cccc|cccccccccccccc}
f_3 & \defl & p_6 &&&  a t_3  & + & a_3 t_2  & + & a_4 t_1  & +& a_5b_1 & + &  a_6 t_0 & + & \ldots \\
f_2 & \defl & p_5 &&&  a t_2 & + & a_3 t_1 & + & a_4 b_1 &+& a_5 t_0 & + & \ldots \\
f_1 & \defl & p_4 &&&  a t_1 & + & a_3 b_1 & + & a_4 t_0 & + & \ldots \\
f_0 & \defl & p_2 &&&  a t_0 & + & \ldots \\
 \end{array}
\]
 where $t_0=b_0$, $t_1=b_2$, $t_2=b_3$, $t_3=b_4$ and $a=a_2$. The set $T$ is given by $S \setminus \{a_0,a_1\} \cup \{t_0, t_1, t_2, t_3\}$. Note in particular that $b_1$, $b_5$, and $b_6$ are elements of $T$. The observation about the $b_j$ we noted above shows that the conditions of Lemma~\ref{lemma:localization_isomorphism} are satisfied for the ring $A=R[T]$. Moreover, we have
 \[
C_{\ell,i-1} \cong R[T][t_0,t_1,t_2,t_3] \slash f_0, f_1, f_2, f_3 \smash{\rlap{,}}
 \] 
 so the lemma implies that $(C_{\ell,i-1})_{a_i} \cong R[T]_{a_i}$ in this particular case.
 
 For the general case, if $\ell \leq k \leq n$ we set
 \[
 t_m \defl \left\{ \begin{array}{ll}
 b_{\ell + m -i} & \mbox{if $m < k-\ell$} \\
 b_{\ell + m+1-i} & \mbox{if $m \geq k-\ell$}
 \end{array} \right.
 \quad \text{and} \quad
 f_m \defl \left\{ \begin{array}{ll}
 p_{\ell + m} & \mbox{if $m < k-\ell$} \\
 p_{\ell + m+1} & \mbox{if $m \geq k-\ell$}
\end{array}  \right.
 \]
 and we let $T \defl S \setminus \{a_0, \ldots, a_{i-1}\} \cup \{t_0, \ldots, t_{n-\ell-1} \}$. Since each $t_j$ occurs only in the monomial $a_i t_j$ in the polynomial $f_j$ and does not occur in any of the $f_{m}$ for $m<j$, we can apply Lemma~\ref{lemma:localization_isomorphism} to the ring $A=R[T]$. The isomorphism
 \[
 C_{\ell,i-1} \cong R[T][t_0,\ldots, t_{n-\ell-1}] \slash f_0, \ldots, f_{n-\ell-1}
 \]
 shows that $(C_{\ell,i-1})_{a_i} \cong R[T]_{a_i}$, so this ring is indeed a unique factorization domain.
 
 If $k< \ell$ or $k > n$, the translation is a little bit simpler.  We can set $t_m \defl b_{\ell +m-i}$, $f_m \defl p_{\ell+m}$, and $T \defl S \setminus \{a_0, \ldots, a_{i-1}\} \cup \{t_0, \ldots, t_{n-\ell} \}$. By the same reasoning as above, Lemma~\ref{lemma:localization_isomorphism} is applicable, so there exists an isomorphism $(C_{i-1})_{a_i} \cong R[T]_{a_i}$. This shows that in all cases, the ring $(C_{\ell,i-1})_{a_i}$ is a unique factorization domain. Nagata's Criterion therefore implies that $C_{\ell,i-1}$ is a noetherian unique factorization domain for each $0\leq i \leq \ell$. This concludes the proof of the inductive step, and so establishes the claim made at the beginning of the proof.
 
 We find in particular that $B^R_{r,k,n}=C_{0,-1}$ is a noetherian unique factorization domain. Moreover, the sequence
 \[
  p_n, p_{n-1}, \ldots, \widehat{p_k}, \ldots, p_0
 \]
 is a regular sequence in $R[S]$ since $R[S] \slash   p_n, \ldots, \widehat{p_k}, \ldots, p_\ell \cong C_{\ell,-1}$ is a domain and in the proof above we have shown that for $\ell \neq k$, the element $\overline{p_{\ell}} \in C_{\ell+1,-1}$ is irreducible, hence in particular not equal to zero.
\end{proof}

 We are now ready to prove Theorem~\ref{thm:test_set_of_complete_intersections}.
 
 \begin{cit}[Theorem~\ref{thm:test_set_of_complete_intersections}]
  There exists a countable test set $\ca{T}$ for $(\mathrm{H})$ with the following properties:
 \begin{enumerate}
 \item[(i)] Each $A \in \ca{T}$ is $\mathbb{N}$-graded, with degree $0$ part $A_0=\mathbb{Z}$;
 \item[(ii)] Each $A \in \ca{T}$ is a global complete intersection ring;
 \item[(iii)] Each $A \in \ca{T}$ is a unique factorization domain. 
 \end{enumerate}
 \end{cit}
 
 \begin{proof}[Proof of Theorem~\ref{thm:test_set_of_complete_intersections}]
 Let $\ca{T}$ be the set of commutative rings
 \[
 \ca{T} \defl \{ B_{r,k,n}^{\mathbb{Z}} \vert r \geq 2, k,n \in \mathbb{N} \}
 \]
 of Definition~\ref{dfn:building_blocks}. Let $\ca{L}$ denote the class of commutative local rings and let $\ca{H}$ be the class of commutative Hermite rings. The implication $(iii) \Rightarrow (i)$ of Theorem~\ref{thm:test_set} shows that the implication
 \[
 (\mathrm{H}[\ca{T}]) \Rightarrow (\mathrm{H}[\ca{L}])
 \]
 holds (see Definition~\ref{dfn:test_set_for_H} for the notation). From Quillen's Patching Theorem it follows that the implication
 \[
 (\mathrm{H}[\ca{L}]) \Rightarrow (\mathrm{H}[\ca{H}])
 \]
 also holds. This shows that $\ca{T}$ is a test set for $(\mathrm{H})$.
 
 Conditions~(ii) and (iii) of Theorem~\ref{thm:test_set_of_complete_intersections} follow directly from Theorem~\ref{thm:building_blocks_ufd}. Finally we note that each polynomial $p_{\ell}$ appearing in the definition of $B^{\mathbb{Z}}_{r,k,n}$ is homogeneous of degree $2$. Thus the quotient ring $B^{\mathbb{Z}}_{r,k,n}$ inherits an $\mathbb{N}$-grading from the polynomial algebra $\mathbb{Z}[S_{r,n}]$, which we endow with the $\mathbb{N}$-grading where each variable has degree $1$ and the constant polynomials have degree $0$. From this convention it immediately follows that the degree $0$ part of $B^{\mathbb{Z}}_{r,k,n}$ is $\mathbb{Z}$.
 \end{proof}
 
 \begin{prop}\label{prop:graded_homotopy_trick}
 Assume that the Hermite ring conjecture $(\mathrm{H})$ holds and let $A=A_0 \oplus A_1 \oplus \ldots$  be an $\mathbb{N}$-graded commutative ring. Then every finitely generated stably free $A$-module is extended from $A_0$.
 \end{prop}
 
 \begin{proof}
 Let $F \colon \CRing \rightarrow \Set$ be the functor which sends a commutative ring $R$ to the set of isomorphism classes of finitely generated stably free $R$-modules. For each maximal ideal $\mathfrak{m}$ of $A$, the local ring $A_{\mathfrak{m}}$ is Hermite, so each finitely generated stably free $A_{\mathfrak{m}}[t]$-module is free by Assumption~$(\mathrm{H})$. From Quillen's Patching Theorem it follows that $F(A) \rightarrow F(A[t])$ is surjective. Since this is a split injection, it is in fact a bijection. 
 
 The Swan--Weibel Homotopy Trick implies that $F(A_0) \rightarrow F(A)$ is a bijection as well. This is proved in \cite[Theorem~V.3.9]{LAM} in the case where $F$ takes values in abelian groups, but the proof works verbatim for functors valued in sets.
 \end{proof}

 We have an analogous result concerning Suslin's Problem about unimodular rows. Let $\ca{C}$ be a class (or set) of commutative rings. We write $(\mathrm{Su}_r[\ca{C}])$ for the assertion:
 \begin{center}
 For all $A \in \ca{C}$, all unimodular rows in $\mathrm{Um}_{r+1}(A[t],(t))$ are completable.
 \end{center}
 
 In \cite[Problem~4]{SUSLIN}, Suslin asked if $(\mathrm{Su}_r[A])$ holds for each commutative local ring $A$ in which $r!$ is a unit. By the equivalence $(i) \Leftrightarrow (ii)$ of Theorem~\ref{thm:test_set}, this is equivalent to the question if $(\mathrm{Su}_r[\CAlg_{\mathbb{Z}[1 \slash r!]}])$ holds, where $\CAlg_{\mathbb{Z}[1 \slash r!]}$ denotes the class of all commutative $\mathbb{Z}[1 \slash r!]$-algebras.
 
 We call a set $\ca{T}$ of commutative rings a \emph{test set for $(\mathrm{Su}_r)$} if the implication
 \[
(\mathrm{Su}_r[\ca{T}]) \Rightarrow (\mathrm{Su}_r[\CAlg_{\mathbb{Z}[1 \slash r!]}])
 \]
 holds.
 
 \begin{thm}
 Let $r \geq 2$ be a natural number. There exists a test set $\ca{T}_r$ for $(\mathrm{Su}_r)$ with the following properties:
 \begin{enumerate}
 \item[(i)] Each $A \in \ca{T}_r$ is $\mathbb{N}$-graded, with degree $0$ part $A_0=\mathbb{Z}[1 \slash r!]$;
 \item[(ii)] Each $A \in \ca{T}_r$ is a quotient $\mathbb{Z}[1 \slash r!][t_1, \ldots, t_n] \slash I$ where $I$ is a complete intersection ideal;
 \item[(iii)] Each $A \in \ca{T}_r$ is a unique factorization domain.
 \end{enumerate}
 \end{thm}
 
 \begin{proof}
 Let $\ca{T}_r$ be the set
 \[
 \ca{T}_r \defl \{B^{\mathbb{Z}[1 \slash r!]}_{r,k,n} \vert k,n \in \mathbb{N} \}
 \]
 of Definition~\ref{dfn:building_blocks}. The equivalence $(ii) \Leftrightarrow (iii)$ of Theorem~\ref{thm:test_set} shows that $\ca{T}_r$ is a test set for $(\mathrm{Su}_r)$. Theorem~\ref{thm:building_blocks_ufd} shows that Conditions~(ii) and (iii) hold. Condition~(i) follows from the fact that each polynomial $p_{\ell}$ in the definition of $B^{\mathbb{Z}[1 \slash r!]}_{r,k,n}$ is homogeneous of degree $2$.
 \end{proof}